\documentclass[12 pt]{amsart}
\newtheorem{theorem}{Theorem}[section]
\newtheorem{lemma}[theorem]{Lemma}
\newtheorem{proposition}[theorem]{Proposition}

\newtheorem{definition}[theorem]{Definition}

\theoremstyle{definition}
\newtheorem{remark}[theorem]{Remark}

\numberwithin{equation}{section}
\newcommand{\CC}{\mathbb{C}}

\newcommand{\ZZ}{\mathbb{Z}}
\newcommand{\DD}{\mathbb{D}}
\renewcommand{\Re}{\mathfrak{Re}}
\renewcommand{\Im}{\mathfrak{Im}}
\newcommand{\HH}{\mathbf{H}}
\newcommand{\NN}{\mathbb{N}}

\DeclareMathOperator{\spa}{span}

\DeclareMathOperator{\ran}{ran}
\title[Schur analysis and discrete analytic functions]{Schur analysis and discrete analytic functions: Rational functions and co-isometric realizations}

\author[D.  Alpay]{Daniel Alpay}
\address{(DA)
Faculty of Mathematics, Physics, and Computation\\
Schmid College of Science and Technology\\
Chapman University\\
One University Drive
Orange, California 92866\\
USA}
\email{alpay@chapman.edu}

\author[D.  Volok]{Dan Volok}
\address{(DV) Mathematics Department\\ Kansas State University\\ 138 Cardwell Hall\\1228 N.  17th Street\\
  Manhattan, KS 66506, USA}
\email{danvolok@math.ksu.edu}

\begin{document}

\maketitle

\begin{abstract}
We define and study rational discrete analytic functions and prove the existence of a coisometric realization for discrete analytic Schur multipliers.
\end{abstract}
\mbox{}\\



\date{today}
\tableofcontents
\section{Introduction}
Schur analysis originates with the work of Schur \cite{schur, schur2} (see \cite{hspnw} for some of the original works)
and in particular with the Schur algorithm, and can be defined as a collection of problems pertaining to functions analytic and contractive in
the open unit disk $\mathbb D$, with a view on applications to operator theory and the theory of linear systems. Often motivated by the latter, counterparts
of Schur analysis have been considered in a number of domains, such as upper triangular operators
(corresponding to linear time-varying systems; see \cite{MR93b:47027,MR94b:93027}), the polydisk (corresponding to $ND$-systems), the theory of slice hyperholomorphic functions \cite{zbMATH06658818}, and many others.
In each version of Schur analysis in a new setting, some aspects of the classical setting will not go through, providing
the peculiarities of the case under consideration.\smallskip

In the present paper we consider the discrete analytic (discrete analytic for short) functions framework.  The discrete analytic counterparts of the Hardy space and the Schur class  were introduced in  \cite{alpay2021discrete}.  Also in that work it was shown that, much like in the classical case,  discrete analytic Schur functions play an important role in interpolation problems.  The next natural steps  include characterization of the related de Branges - Rovnyak spaces in terms of translation-invariance (a Beurling - Lax type theorem; see Theorem \ref{beurlax}), coisometric realizations of discrete analytic Schur functions that establish a connection with linear discrete time system theory (Theorem \ref{schreal}),  and study of  rational discrete analytic functions.  \smallskip

The notion of rationality in the setting of discrete analytic functions is relative:  it depends on how one defines  products of functions. The usual pointwise product does not preserve discrete analyticity. This led to the conjecture of Isaacs in \cite{MR0052526} that every rational discrete analytic function is a polynomial.  Isaacs' notion of discrete analyticty was slightly different from the one used here,  and Isaacs' conjecture was ultimately proved false by Harman \cite{MR0361111}.  However,  these works indicate that the class of discrete analytic functions, that are rational in the sense of the pointwise product, is not rich enough for applications.\smallskip

Another notion of discrete analytic rationality,  based on  Cauchy - Kovalevskaya extension of the pointwise product from the real axis,  similar to Cauchy - Kovalevskaya product of hyperholomorphic functions in the quaternionic settting,  was introduced and studied in \cite{ajsv}.   \smallskip
 
Our current approach is based on  the convolution product $\odot$ of discrete analytic polynomials with respect to a certain "standard" basis $z^{(n)}$.  Several different bases of discrete analytic polynomials can be found in the literature.  For instance, a different polynomial basis suitable for "power" series expansions was considered by Zeilberger in \cite{zeil}.  The basis $z^{(n)},$ that previously  appeared in \cite{alpay2021discrete} and,  in a slightly different form,  in \cite{ajsv},  is closely related to eigenfunctions of elementary difference operators on the integer lattice, which makes it expedient from the computational point of view.  The corresponding class of rational discrete analytic functions turns out to be quite rich; it includes functions that may fail to have a power series expansion in terms of $z^{(n)}.$ This is achieved by utilizing a certain representation of a rational function,  related to transfer functions of discrete time linear systems (the so-called system realizations of rational functions). As an example of the differences between the discrete analytic setting and the classical case,  the reproducing kernel $K(z,w)$ of the discrete analytic Hardy space turns out to be rational as a function of $z$, but its McMillan degree depends on $w$.\smallskip

  The paper consists of five sections beside this introduction, and we give now its outline.  A brief survey of past results and some preliminary computations are presented in Sections \ref{eigensec} and \ref{reviewsec}.  In
  particular, some new identities pertaining to the two elementary difference operators are proved.  In Section \ref{rationalsec} we analyze the notion of discrete analytic  rationality based on the convolution product.  The main results of the paper are given in Section \ref{realsec}. Using the theory of linear relations, we construct coisometric realizations for discrete analytic Schur functions and characterize the associated de Branges - Rovnyak spaces in terms of translation-invariance.  Finally,  Section \ref{meshsec} gives some motivation by considering the case of an arbitrary mesh,  and letting the
  mesh go $0$.

  \section{Eigenfunctions of difference operators and a polynomial basis}
  \label{eigensec}

We recall some notations and definitions to set the framework. The symbol $\Lambda$ denotes the integer lattice $\mathbb Z+i\mathbb Z$. 
A function $f:\Lambda\longrightarrow \mathcal{V},$ where $\mathcal V$ is a complex vector space,  is said to be {\em discrete analytic} if 
\begin{equation}\label{ferrand}
\dfrac{f(z+1+i)-f(z)}{1+i}=\dfrac{f(z+1)-f(z+i)}{1-i},\quad z\in\Lambda.
\end{equation}
This definition originates with  Ferrand \cite{MR0013411}.  It can be re-formulated in terms of discrete integration as follows.  The discrete integral of function $f:\Lambda\longrightarrow\mathcal{V}$ over path $\gamma=(z_0,z_1,\dots, z_N)$ is defined by
    \begin{equation}
      \label{di}\int_\gamma f\delta z=\sum_{n=1}^N \dfrac{f(z_{k-1})+f(z_k)}{2}(z_k-z_{k-1}),
\end{equation}
 A function $f(z)$ is discrete analytic if,  and only if,  all the discrete integrals on closed paths are equal to $0$. \smallskip

In terms of the difference operators 
\begin{gather*}
\delta_xf (z):=f(z+1)-f(z),\\ \delta_yf(z):=f(z+i)-f(z),
\end{gather*}
equation \eqref{ferrand} can be re-written as
$$\bar D f=0,$$
where 
\begin{gather}
\label{dirac}
\bar D=\alpha_-\delta_x+\alpha_+\delta_y+\dfrac{1}{2}\delta_x\delta_y,\\
\alpha_\pm=\dfrac{1\pm i}{2}.
\end{gather}
The operator $\bar D$ plays the role of the Wirtinger derivative $\frac{\partial}{\partial \bar z}$; in particular,  $\bar D$ and its complex conjugate
\begin{equation} 
D=\alpha_+\delta_x+\alpha_-\delta_y+\dfrac{1}{2}\delta_x\delta_y
\end{equation}
satisfy
$$D\bar D=\bar D D=\Delta,$$
where $\Delta$ is a discrete Laplacian (corresponding to a diagonal lattice).  

We  define the derivative of a discrete analytic function $f(z)$ to be $-\frac{i}{2}{D}f$.

\begin{lemma}
Assume that $f:\Lambda\longrightarrow\mathcal{V}$ is a discrete analytic function.  Its derivative is given by
\begin{equation}
\label{dafdelta}
   -\frac{i}{2} {D}f=\frac{1}{2}(\delta_x-\delta_y)f
    \end{equation}
\end{lemma}  
\begin{proof}
  The result follows from
  \begin{equation}
{D}-\bar D=(\alpha_+-\alpha_-)(\delta_x-\delta_y)=i(\delta_x-\delta_y).
\end{equation} 
\end{proof}

We shall
denote by $\mathbf{H}(\Lambda,\mathcal{V})$  the space of all discrete analytic functions $f:\Lambda\longrightarrow\mathcal{V};$
$$\mathbf{H}(\Lambda,\mathcal{V})=\ker(\bar D).$$
The following observation turns out to be quite useful.
\begin{proposition} \label{propfour}
The operators $I+\delta_x,$ $I+\delta_y,$ $I+\alpha_\pm\delta_x,$ $I+\alpha_\pm \delta_y$ are invertible on $\mathbf{H}(\Lambda,\mathcal{V});$
\begin{gather}\label{four1}
I+\alpha_+\delta_y=(I+\alpha_-\delta_x)^{-1};\\
I+\alpha_-\delta_y=(I+\delta_x)(I+\alpha_-\delta_x)^{-1};\\
I+\alpha_+\delta_x=(I+\delta_y)(I+\alpha_+\delta_y)^{-1}.\label{four3}
\end{gather}
\end{proposition}
\begin{proof} Since $\delta_x$ and $\delta_y$ commute (and, therefore,  commute with $\bar D$),  the invertibility of the translation operators $I+\delta_x$ and $I+\delta_y$  on $\mathbf{H}(\Lambda,\mathcal{V})$ is immediate. The rest follows from the identities:
\label{four}
\begin{gather*}
(I+\alpha_-\delta_x)(I+\alpha_+\delta_y)=I+\bar D;\\
(I+\alpha_-\delta_x)(I+\alpha_-\delta_y)=I+\delta_x-i\bar D;\\
(I+\alpha_+\delta_x)(I+\alpha_+\delta_y)=I+\delta_y+i\bar D.
\end{gather*}
\end{proof}

Next,  we turn to eigenfunctions of the difference operators $\delta_x$ and $\delta_y.$

\begin{proposition} \label{eigenprop} The set of complex eigenvalues of $\delta_x$,  as a linear operator on $\mathbf{H}(\Lambda,\CC)$,
is $\CC\setminus\{-1,-2\alpha_\pm\}.$ Each eigenspace of $\delta_x$ is one-dimensional; the eigenspace corresponding to an eigenvalue $\lambda$ is spanned by the function
\begin{equation}\label{genela}e_\lambda(z)=(1+\lambda)^{\Re(z)}\left(\dfrac{1+\alpha_+\lambda}{1+\alpha_-\lambda}\right)^{\Im(z)}.\end{equation}
The set of complex eigenvalues of $\delta_y$, as a linear operator on $\mathbf{H}(\Lambda,\CC)$, is the same as for $\delta_x;$ each eigenspace of $\delta_y$ is one-dimensional;  the eigenspace corresponding to an eigenvalue $\mu$ is spanned by the function
$ e_\lambda(z),$
where
$$\lambda=-\dfrac{i\mu}{1+\alpha_+\mu}.$$
\end{proposition}

\begin{proof} Let $\lambda\in\CC,$ and let $e_\lambda(z)\not\equiv 0$ be a discrete analytic function, such that
$$(\delta_xe_\lambda)(z)=\lambda e_\lambda(z),\quad z\in\Lambda.$$
Then $$(I+t\delta_x)e_\lambda=(1+t\lambda)e_\lambda,\quad t\in\CC,$$ and Proposition \ref{propfour} implies that
$\lambda\not\in\{-1,-2\alpha_\pm\}.$  In particular, 
$$(I+\delta_x)e_\lambda=(1+\lambda)e_\lambda,$$
so
\begin{equation}\label{111}e_{\lambda}(z+1)=(1+\lambda)e_\lambda(z),\quad z\in\Lambda.\end{equation}
Equation \eqref{111} determines a discrete analytic function up to a multiplicative constant. Indeed,  in view of \eqref{four1} and \eqref{four3},
$$(I+\delta_y)e_\lambda=(I+\alpha_+\delta_x)(I+\alpha_-\delta_x)^{-1}e_\lambda=\dfrac{1+\alpha_+\lambda}{1+\alpha_-\lambda}e_\lambda,$$
hence
\begin{equation}\label{112}
e_\lambda(z+i)=\dfrac{1+\alpha_+\lambda}{1+\alpha_-\lambda}e_\lambda(z),\quad z\in\lambda.
\end{equation}
Combining \eqref{111} and \eqref{112} with $e_\lambda(0)=1$ leads to \eqref{genela}.

The analysis of eigenspaces of $\delta_y$ is similar.
\end{proof}
We note that, for every $z\in\Lambda,$ the mapping $\lambda\mapsto e_\lambda(z)$ is analytic in the open unit disk $\DD$.  It admits the Taylor expansion
\begin{equation}
e_\lambda(z)=  \sum_{n=0}^\infty \lambda^n z^{(n)},\quad \lambda\in\DD, z\in\Lambda.
  \label{gene}
\end{equation}
The discrete analytic functions $z^{(n)}$ are polynomials of degree $n$ in variables  $\Re(z)$ and $\Im(z),$ satisfying the relation
\begin{equation}\label{si3}\delta_x z^{(n)}=z^{(n-1)},\quad n\in\NN.\end{equation} In particular,
$$z^{(0)}=1\quad\text{and}\quad z^{(1)}=z.$$
Polynomials $z^{(n)}$ have previously appeared in \cite{alpay2021discrete},  where it was shown that they form a basis of discrete analytic polynomials,  and that they can be defined recursively by means of an operator $Z$ that acts on $\mathbf{H}(\Lambda,\CC):$
\begin{gather}\label{defz1}
z^{(n)}=Z^n 1;\\
  (Zf)(z)=\frac{f(0)-f(z)}{2}+\int_0^zf\delta z,\label{defz2}
  \end{gather}
  where the integral is independent of the chosen path in view of the
assumed discrete analyticity of $f(z).$ We note that  operator $Z$ can be defined by \eqref{defz2} on  discrete analytic functions with values in an arbitrary complex vector space $\mathcal V$, as well, and that, operators $Z$ and $\delta_x$ play the role of the  forward-shift and backward-shift operators in the open unit disk setting:
  \begin{gather}
    \label{structural-identity1}
    (Z\delta_x f)(z)=f(z)-f(0);\\
    \label{structural-identity2}
    (\delta_x Zf)(z)=f(z).
  \end{gather}

We close this section with a generalization of  the Chu-Vandermonde formula
$$\binom{z+w}{n}=\sum_{k=0}^n\binom{z}{k}\binom{w}{n-k},$$ where $z,w,n\in\ZZ_+$ and
it is understood that $\binom{z}{n}=0$ for $n>z.$ Note that for  $z\in\ZZ_+ $ the expansion \eqref{gene} implies that
$$z^{(n)}=\binom{z}{n}.$$
\begin{proposition}
  Let $z,w\in\Lambda$ and $n\in\ZZ_+.$  It holds that
  \begin{equation}
    \label{chu}
(z+w)^{(n)}=\sum_{k=0}^nz^{(k)}w^{(n-k)}
    \end{equation}
\end{proposition}  
\begin{proof}
Let $w=u+vi\in \Lambda$ be fixed.  Consider function $f:\Lambda\longrightarrow\CC$ given by
$$f(z)=(z+w)^{(n)}.$$
Then 
$$f(z)=(I+\delta_x)^u(I+\delta_y)^v z^{(n)};$$
in particular, $f(z)$ is discrete analytic.  In view of identity \eqref{si3},  
$$(\delta_x^{n+1} f)\equiv 0,$$
hence, by Proposition \ref{taylorthm}, $f(z)$ admits the power series expansion \eqref{taylor} and is, in fact,  a discrete analytic polynomial of degree at most $n.$
Finally,  for $0\leq k\leq n,$
$$(\delta_x^kf)(z)=(I+\delta_x)^u(I+\delta_y)^v z^{(n-k)}=(z+w)^{(n-k)},$$
and it only remains to set $z=0.$
\end{proof}

\section{Hardy space, Schur functions and convolution product}
\label{reviewsec}
In this section we review the discrete analytic counterparts of the Hardy space and Schur class, introduced in  \cite{alpay2021discrete}.   Although in that work the discrete analytic Schur functions were assumed to have complex values,  there is little difficulty involved in adapting the appropriate results to the case of functions with values that are bounded linear operators between Hilbert spaces.\smallskip

We begin with the
observation that for every  $z$ in the right half-lattice
$$\Lambda_+=\ZZ_++i\ZZ$$  the generating function $e_\lambda(z)$ of the basis $z^{(n)}$  (see \eqref{genela}, \eqref{gene}) is analytic, as a function of $\lambda,$ in the disk $\{|\lambda|<\sqrt{2}\}.$ Hence
\begin{equation}\label{hadamard}\limsup_{n\rightarrow\infty}\sqrt[n]{|z^{(n)}|}=\dfrac{1}{\sqrt{2}},\quad z\in\Lambda_+,\end{equation}
and the series
$$K(z,w)=\sum_{n=0}^\infty z^{(n)}{\bar w}^{(n)}$$
defines a positive\footnote{Note that ${\bar w}^{(n)}=\overline{w^{(n)}}$.} kernel in $\Lambda_+.$ It is for this reason that we shall consider in the sequel the space $\HH(\Lambda_+,\mathcal{V})$ of functions $f:\Lambda_+\longrightarrow\mathcal{V}$ that are discrete analytic on the right half-lattice $\Lambda_+,$ 
rather than on the whole lattice $\Lambda.$ The conclusions of the previous section still hold, with one notable exception (see Proposition \ref{propfour}): on the space
$\mathbf{H}(\Lambda_+,\CC)$ the operators $I+\delta_x$ and $I+\alpha_-\delta_y$ are {\em not} invertible; $\lambda=-1$ is an eigenvalue, of $\delta_x,$  and the corresponding eigenfunction is
$$e_{-1}(z)=\left\{\begin{aligned}(-i)^{\Im(z)},&\quad\text{if }\Re(z)=0,\\0,&\quad\text{if }\Re(z)>0.\end{aligned}
\right.$$
\begin{proposition} \label{propfourmod}
The operators  $I+\delta_y,$ $I+\alpha_+\delta_y$,  $I+\alpha_\pm\delta_x$ are invertible on $\mathbf{H}(\Lambda_+,\mathcal{V});$
\begin{gather}\label{mod1}
I+\alpha_+\delta_y=(I+\alpha_-\delta_x)^{-1};\\
I+\alpha_+\delta_x=(I+\delta_y)(I+\alpha_+\delta_y)^{-1}.\label{mod2}
\end{gather}
\end{proposition}

Polynomials $z^{(n)}$ play the role of the powers of $z$ in power series expansions of discrete analytic functions.  The class of discrete analytic functions that admit such an expansion can be characterized as follows.
\begin{proposition}\label{taylorthm} Let $\mathcal{V}$ be a normed vector space.
  A function $f\in\mathbf{H}(\Lambda_+,\mathcal{V})$  admits a power series expansion in terms of $z^{(n)}$ if,  and only if, 
  \[
\forall z\in\Lambda_+\quad \lim_{n\rightarrow\infty} (Z^n\delta_x^nf)(z)=0.
  \]
 Such an expansion is unique:
\begin{equation}\label{taylor}f(z)=\sum_{n=0}^\infty (\delta_x^n f)(0)z^{(n)},\quad z\in\Lambda_+.\end{equation}
\end{proposition}  

\begin{proof} The uniqueness of the expansion follows from the identity \eqref{si3} and the fact that,  in view of \eqref{defz1} -- \eqref{defz2},  $$0^{(0)}=1\text{ and }
0^{(n)}=0,\quad n\geq 1.$$ The rest is based on the observation that, according to
the  identities  \eqref{structural-identity1} and \eqref{structural-identity2},
  \[
    \begin{split}
      f(z)&=f(0)+(Z\delta_xf)(x)\\
      &=f(0)+(\delta_xf)(0)z^{(1)}+(Z^2\delta_x^2f)(z)\\
      &=f(0)+(\delta_xf)(0)z^{(1)}+(\delta_x^2f)(0)z^{(2)}+(Z^3\delta_x^3f)(z)\\
      &\hspace{2.5mm}\vdots\\
      &=\sum_{n=0}^{N-1} (\delta_x^nf)(0)z^{(n)}+(Z^N\delta_x^Nf)(z),
      \end{split}
    \]
    and hence the result by letting $N\rightarrow\infty$.
 \end{proof} 

\begin{lemma}\label{zzero} Let  $\mathcal{V}$ be a normed vector space,  and let $f\in\mathbf{H}(\Lambda_+,\mathcal{V}).$ Then 
$$\forall z\in\Lambda_+\quad \lim_{n\rightarrow\infty}(Z^nf)(z)=0.$$
\end{lemma}
\begin{proof}
The proof is by induction on $|z|.$ If $|z|=0,$ then $(Z^nf)(z)=0$ for any $n\in\NN.$ Let us  now fix
$z\in\Lambda_+$ and assume that
$$\lim_{n\rightarrow\infty}(Z^nf)(z)=0.$$
Taking into account  Proposition \ref{propfourmod} and identity \eqref{structural-identity2}, we observe that
\begin{gather*}
(Z^nf)(z+1)=((I+\delta_x)Z^nf)(z)=(Z^nf)(z)+(Z^{n-1}f)(z);\\
\begin{split}(Z^nf)(z+i)&=((I+\delta_y)Z^nf)(z)=((I+\alpha_+\delta_x)(I+\alpha_-\delta_x)^{-1}Z^nf)(z)\\&=\sum_{k=0}^n(-\alpha_-)^k((Z^{n-k}f)(z)+\alpha_+(Z^{n-k-1}f)(z)).
\end{split}\\
\begin{split}(Z^nf)(z-i)&=((I+\delta_y)^{-1}Z^nf)(z)=((I+\alpha_+\delta_x)^{-1}(I+\alpha_-\delta_x)Z^nf)(z)\\&=\sum_{k=0}^n(-\alpha_+)^k((Z^{n-k}f)(z)+\alpha_-(Z^{n-k-1}f)(z)).
\end{split}
\end{gather*}
Thus $(Z^nf)(z+1)$ tends to $0,$ and, since $|\alpha_\pm|<1,$ so do $(Z^nf)(z\pm i).$
\end{proof}

Given a Hilbert space $\mathfrak{H}, $ we consider the reproducing kernel Hilbert space  $\mathbf{H}_2(\Lambda_+,\mathfrak{H})$ with the reproducing kernel $K(z,w)I_{\mathfrak{H}}.$  It  can be characterized as the space of discrete analytic functions $$f:\Lambda_+\longrightarrow\mathfrak{H}$$ that  admit power series expansions in terms of $z^{(n)}$ with square-summable coefficients:
\begin{gather*}
f(z)=\sum_{n=0}^\infty \hat f(n)z^{(n)},\quad z\in\Lambda_+;\\
\|f\|_2^2:=\sum_{n=0}^\infty\|\hat f(n)\|_{\mathfrak{H}}^2<\infty.
\end{gather*}

The space $\mathbf{H}_2(\Lambda_+,\mathfrak{H})$ is the discrete analytic  counterpart of the classical Hardy space of the open unit disk; in particular, on this Hilbert space the forward-shift operator $Z$ is adjoint to the backward-shift operator $\delta_x.$ The discrete analytic counterpart of the Schur class is defined as follows.

\begin{definition} Let $\mathfrak{H}_1$ and $\mathfrak{H}_2$ be Hilbert spaces.  A $\mathbf L(\mathfrak{H}_1,\mathfrak{H}_2)$-valued function $S(z),$  discrete analytic in $\Lambda_+,$ is said to be a
  discrete analytic Schur function if 
$$\sum_{n=0}^\infty \|(Z^nS)(z)\|^2_{\mathfrak{H}_2}<\infty,\quad z\in\Lambda,$$ and the kernel
$$K^S(z,w)=\sum_{n=0}^\infty z^{(n)}{\bar w}^{(n)}I_{\mathfrak{H}_2}-(Z^n S)(z)(Z^nS)(w)^{*}$$
is positive in $\Lambda_+$.  The corresponding reproducing kernel Hilbert space is called the de Branges - Rovnyak Hilbert space associated with $S(z);$ it will be denoted by $\mathcal{H}(S).$
\end{definition}

Discrete analytic Schur functions act on the Hardy space via a suitably defined convolution product.  The proof of the following theorem follows the proof of \cite[Theorem 5.2]{alpay2021discrete} and will be omitted.

\begin{theorem} Let $S(z)$ be a $\mathbf{L}(\mathfrak{H}_1,\mathfrak{H}_2)$-valued discrete analytic Schur function.  Then 
$S(z)$ admits a power series expansion
\begin{equation}\label{schtay}S(z)=\sum_{n=0}^\infty \hat S(n)z^{(n)},\quad z\in\Lambda_+,\end{equation}
and 
for every $f(z)$ in the Hardy space $\HH_2(\Lambda_+,\mathfrak{H}_1)$ the convolution product
$$(S\odot f)(z)=\sum_{n=0}^\infty\sum_{m=0}^n \hat S(m)\hat f(n-m)z^{(n)}$$
converges pointwise on $\Lambda_+$ to a discrete analytic function from the Hardy space $\HH_2(\Lambda_+,\mathfrak{H}_2)$.  Moreover,
 the linear operator
$$M_Sf=S\odot f$$
is a contraction from  $\HH_2(\Lambda_+,\mathfrak{H}_1)$ to $\HH_2(\Lambda_+,\mathfrak{H}_2)$. 

Conversely,  if $S:\Lambda_+\longrightarrow\mathbf{L}(\mathfrak{H}_1,\mathfrak{H}_2)$ is a discrete analytic function,  that admits a power series expansion \eqref{schtay} and is such that the multiplication operator $M_S$ is a contraction 
 from  $\HH_2(\Lambda_+,\mathfrak{H}_1)$ to $\HH_2(\Lambda_+,\mathfrak{H}_2)$, then $S(z)$ is a discrete analytic Schur function.
\end{theorem}

\begin{proposition}\label{dbrrov} Let $S(z)$ be a $\mathbf{L}(\mathfrak{H}_1,\mathfrak{H}_2)$-valued discrete analytic Schur function.  The associated de Branges - Rovnyak space $\mathcal{H}(S)$ is the space
$$\ran\sqrt{I-M_SM_S^*},$$
equipped with the range norm.  In particular,  $\mathcal{H}(S)$ is contractively included in the Hardy space $\HH_2(\Lambda_+,\mathfrak{H}_2).$
\end{proposition}
\begin{proof} The statement follows from the fact that
$$K^S(\cdot,w)\alpha=(I-M_SM_S)^*(K(\cdot,w)\alpha),\quad \alpha\in\mathfrak{H}_2.$$
\end{proof}

We close this section with a short discussion of the convolution product.  Since not every discrete analytic function admits a power series expansion,  it is clear that such a product is not always defined.  Note,  however, that one can  define 
the convolution product of a $\mathbf{L}(\mathfrak{H}_1,\mathfrak{H}_2)$-valued discrete analytic polynomial 
$$p(z)=\sum_{n=0}^N \hat p(n)z^{(n)}$$
with any  $\mathbf{L}(\mathfrak{H}_3,\mathfrak{H}_1)$-valued discrete analytic function $f(z)$  by
$$(p\odot f)(z)=\sum_{n=0}^N \hat p(n)(Z^nf)(z).$$
Similarly,  if $g(z)$ is a $\mathbf{L}(\mathfrak{H}_2,\mathfrak{H}_3)$-valued discrete analytic function, then we set
$$(g\odot p)(z)=\sum_{n=0}^N (Z^ng)(z)\hat p(n).$$

\begin{theorem}\label{resolve} Let $A\in\mathbf{L}(\mathfrak{H})$ and let $p(z)=I-zA.$
There exists a discrete analytic function 
$f:\Lambda_+\longrightarrow\mathbf{L}(\mathfrak{H}),$
such that
\begin{equation}\label{lid}(p\odot f)(z)=I,\quad z\in\Lambda_+,\end{equation}
if,  and only if, 
\begin{equation}\label{excl}\{-2\alpha_\pm\}\cap\sigma(A)=\emptyset.\end{equation}
In this case such a function $f(z)$ is unique,  and satisfies
\begin{equation}\label{rid} (f\odot p)(z)=I,\quad z\in\Lambda_+.\end{equation}
It is given by
\begin{equation}\label{resA}f(z)=(I-zA)^{-\odot}:=e_A(z)=(I+A)^{\Re(z)}(I+\alpha_+A)^{\Im(z)}(I+\alpha_-A)^{-\Im(z)}.\end{equation}
\end{theorem}
\begin{proof} 
Suppose that $f$ is  a discrete analytic  function, satisfying \eqref{lid}. Then
$$f=I+Z(Af),$$
hence $f(0)=I,$ and
\begin{gather*}(\delta_xf)(z)=Af(z),\quad z\in\Lambda_+;\\
((I+\alpha_\pm \delta_x)f)(z)=(I+\alpha_\pm A)f(z),\quad z\in\Lambda_+;\\
I=f(0)=(I+\alpha_\pm A)((I+\alpha_{\pm}\delta_x)^{-1}f)(0),
\end{gather*}
where we used Proposition \ref{propfourmod} to obtain the last equality.
Thus the necessity of the condition \eqref{excl} is established.\smallskip

Furthermore,  since for any $z\in\Lambda_+$ we have
\begin{gather*}
f(z+1)=((I+\delta_x)f)(z)=(I+A)f(z),\\
\begin{split}f(z+i)=((I+\delta_y)f)(z)&=((I+\alpha_+\delta_x)(I+\alpha_-\delta_x)^{-1}f)(z)\\&=(I+\alpha_+A)(I+\alpha_-A)^{-1}f(z),\end{split}
\end{gather*}
(see \eqref{mod2}, 
formula \eqref{resA} follows.  Since $f(z)$ commutes with $A,$ equality \eqref{rid} holds, as well. 
The converse direction is proved by reversing the above argument.
\end{proof}

\begin{remark}
If the spectral radius $\rho(A)<\sqrt{2},$ then
$$(I-zA)^{-\odot}=\sum_{n=0}^\infty z^{(n)}A^n.$$
This expansion also holds for any $A$ satisfying \eqref{excl} and $z\in\ZZ_+.$
\end{remark}

\section{Rational discrete analytic functions}
\label{rationalsec}
 In this section we seek to characterize $\odot$-quotients of discrete analytic polynomials, which we shall call {\em rational discrete analytic  functions,} in terms of system realizations.

\begin{definition} A discrete analytic function $f:\Lambda_+\longrightarrow\CC^{m\times n}$ is said to be rational if it can be represented as
\begin{equation}\label{realization}f(z)=D+C(I-zA)^{-\odot}\odot (zB),\end{equation}
where $A,B,C,D$ are complex matrices of suitable dimensions,  and 
$$\{-2\alpha_\pm\}\cap\sigma(A)=\emptyset.$$
\end{definition}

\begin{remark} A rational discrete analytic function may not admit a power series expansion in terms of $z^{(n)}$.\end{remark}

\begin{proposition} \label{ratmap} The linear mapping
$$T:D+C(I-zA)^{-\odot}\odot(zB)\mapsto D+zC(I-zA)^{-1}B$$
 establishes a $1$-to-$1$ correspondence between rational discrete analytic functions and rational functions of the complex variable $z$,  that have no poles in $\{0,-\alpha_\pm\}.$
\end{proposition}
\begin{proof}
For a rational discrete analytic function $f(z)$ of the form \eqref{realization},  
matrices $D$ and $CA^{k-1}B,$ $k\in\NN,$ are determined by $f$ as follows:
$$D=f(0),\quad CA^{k-1}B=(\delta_x^k f)(0).$$ Indeed,  since $\delta_xZ=I$ and
$$\delta_x(I-zA)^{-\odot}=(I-zA)^{-\odot}A,$$
one has
$$\delta_x^kf=C(I-zA)^{-\odot}A^{k-1}B.$$
Thus 
$$(Tf)(z)=\sum_{n=0}^\infty (\delta_x^n f)(0)z^n$$
is independent of the choice of representation \eqref{realization};
the mapping $T$ is well-defined and surjective.  The linearity of $T$ is clear from the formula
\begin{multline*}
D_1+C_1(I-zA_1)^{-\odot}\odot(zB_1)+D_2+C_2(I-zA_2)^{-\odot}\odot(zB_2)\\
=D_1+D_2+\begin{pmatrix}C_1&C_2\end{pmatrix}\left(I-z\begin{pmatrix}A_1&0\\0&A_2\end{pmatrix}\right)^{-\odot}\odot
\left(z\begin{pmatrix}B_1\\B_2\end{pmatrix}\right).
\end{multline*}
 The injectivity of $T$ can be established as follows.  Consider a rational function
$$D+C(I-zA)^{-\odot}\odot(zB)$$
in the kernel of $T,$ where $A\in\CC^{m\times m}$  satisfies  \eqref{excl}. From the Maclaurin expansion of $(Tf)(z)$ it follows that
$$D=0\text{ and } CA^{k-1}B=0,\quad k\in\NN.$$
 if $A\in\CC^{m\times m},$ let
$$\mathcal M=\ran(B)+\ran(AB)+\ran(A^2B)+\dots+\ran(A^{m-1}B),$$
-- an $A$-invariant subspace of $\ker(C)$. Since 
  $A$ satisfies  \eqref{excl},  and $\mathcal M$ is a finite-dimensional complex vector space,  $(I+\alpha_\pm A)$ map $\mathcal M$ bijectively to itself.  In view of \eqref{resA},  
$$C(I-zA)^{-\odot}B=0,\quad z\in\Lambda_+,$$
hence
$$C(I-zA)^{-\odot}\odot(zB)=Z(C(I-zA)^{-\odot}B)=0.$$
\end{proof}

\begin{proposition}
  \label{rk-ex}
  The kernel
  \[
K(z,w)=\sum_{n=0}^\infty z^{(n)}\overline{w^{(n)}}
\]
is rational (as a function of $z$) for $w\in\Lambda_+,$ 
and has McMillan degree ${\rm Re}\,w +|{\rm Im}\, w|$. 
\end{proposition}

\begin{proof}
  From the generating function formula \eqref{gene}, we see that the function of the complex variable $t$
  \[
t\mapsto e_t(w)
\]
is rational and regular at the origin, and so there exist matrices $A,B,C$ depending on $w$ such that $w^{(n)}=CA^{n-1}B$, $n=0,1,\ldots$.  To conclude, we note that $K(\cdot,w)=T^{-1}e_\cdot(w)$ and that
the McMillan degree of
$e_\cdot(w)$ is equal to $\Re(w) +|\Im(w)|$.
\end{proof}

Representations \eqref{realization} make it possible to define the convolution product of rational discrete analytic functions of suitable dimensions,  that extends the convolution product of discrete analytic polynomials and corresponds to the pointwise product of rational functions of the complex variable $z.$

\begin{definition}\label{defprod}
If $f_1(z)$ and $f_2(z)$ are rational discrete analytic functions of suitable dimensions,  
$$f_k(z)=D_k+C_k(I-zA_k)^{-\odot}\odot(zB_k),\quad k=1,2,$$
then
$$(f_2\odot f_1)(z):=D+C(I-zA)^{-\odot}\odot(zB),$$
where
\begin{gather*}
A=\begin{pmatrix} A_1&0\\ B_2C_1&A_2\end{pmatrix}, \quad B=\begin{pmatrix} B_1\\ B_2D_1\end{pmatrix};\\
C=\begin{pmatrix} D_2C_1&C_2\end{pmatrix},\quad D=D_2D_1.
\end{gather*}
\end{definition}

If $f_1(z)$ and $f_2(z)$ are rational discrete analytic functions of suitable dimensions,  then 
\begin{equation}\label{convpont}(T(f_1\odot f_2))(z)=(Tf_1)(z)(Tf_2)(z).\end{equation} In particular,  we can describe the $\odot$-inverse of a rational discrete analytic function. 
\begin{proposition}\label{invprop}
Let $f:\Lambda_+\longrightarrow\CC^{n\times n}$ be a discrete analytic rational function;
$$f(z)=D+C(I-zA)^{-\odot}\odot (zB).$$
Denote $$A^\times=A-BD^{-1}C.$$ If $$\det(D)\not=0\quad\text{and}\quad \{-2\alpha_\pm\}\cap\sigma(A^\times)=\emptyset,$$
then the rational discrete analytic function
$$f^{-\odot}(z)=D^{-1}-D^{-1}C(I-zA^\times)^{-\odot}\odot(zBD^{-1}),$$
is the $\odot$-inverse of $f(z):$
$$(f\odot f^{-\odot})(z)=(f^{-\odot}\odot f)(z)=I.$$
\end{proposition}
We are now in a position to compare our notion of discrete analytic rationality to the classical notion of "quotients of polynomials".
\begin{theorem}  A discrete analytic function $f:\Lambda_+\longrightarrow\CC^{m\times n}$ is rational if,  and only if,  there exists a $\CC$-valued  discrete analytic polynomial $p(z)\not\equiv 0,$ such that
$(pI\odot f)(z)$ is a discrete analytic polynomial.
\end{theorem}
\begin{proof}
The "only if" part of the statement follows  immediately from Proposition \ref{ratmap} and identity \eqref{convpont}. As to the "if" part,  in view of Proposition \ref{invprop}, it suffices to show that if we choose,   among  all $\CC$-valued non-zero discrete analytic polynomials $p(z)$ with the property that $(pI\odot f)(z)$ is a discrete analytic polynomial,  a polynomial $p$ with $\deg(p)$ -- minimal, then $(Tp)(z)$ does not vanish on the set $\{0,-\alpha_\pm\}.$ However, if $(Tp)(0)=0,$ then $p=z\odot p_1,$ where $p_1=\delta_x p,$ $\deg(p_1)=\deg(p)-1,$
and $$p_1\odot f=\delta_x(p\odot f)$$
is a discrete analytic polynomial, in contradiction to the minimality of $\deg(p).$

Similarly, if $(Tp)(-\alpha_+)=0,$ then $p=(z+\alpha_+)\odot p_1,$ where $p_1$ is a discrete analytic polynomial,  $\deg(p_1)=\deg(p)-1,$ and
$$(z+\alpha_+)\odot (p_1\odot f)=q,$$
where $q(z)$ is a discrete analytic polynomial.  But then
$$(I+\alpha_+\delta_x)(p_1\odot f)=\delta_x q,
$$
and, since $D(p_1\odot f)=Dq=0,$
$$p_1\odot f+\delta_y(p_1\odot f)=i\delta_y q.$$
Thus $p_1\odot f$ is a discrete analytic polynomial, and again we arrive at a contradiction.\smallskip

Finally,  if $(Tp)(-\alpha_-)=0,$ then $p=(z+\alpha_-)\odot p_1,$ where $p_1$ is a discrete analytic polynomial,  $\deg(p_1)=\deg(p)-1,$ and
$$(z+\alpha_-)\odot (p_1\odot f)=q,$$
where $q(z)$ is a discrete analytic polynomial.  But then
$$(I+\alpha_-\delta_x)(p_1\odot f)=\delta_x q,
$$
and, since $D(p_1\odot f)=Dq=0,$
$$p_1\odot f=-i\delta_y q$$
 is a discrete analytic polynomial, leading to a contradiction in this case, as well.
\end{proof}
As a final result of this section, we shall address the notion of discrete analytic rationality in terms of backward-shift invariance.
\begin{theorem}
A discrete analytic function $f:\Lambda_+\longrightarrow\CC^{m\times n}$ is rational if,  and only if,  
$$\dim(\spa\{\delta_x f,\delta_x^2f,\dots\})<\infty.$$
\end{theorem}
\begin{proof}
Much like in the proof of the previous theorem,  the "only if" part of the statement follows immediately from Proposition \ref{ratmap}, and it suffices to verify that if one chooses a minimal $n\in\NN,$ such that
$$\delta_x f,\delta^2_x f,\dots,\delta_x^{n+1}f$$
are linearly dependent,  then
$$F(z):=\begin{pmatrix} \delta_x f(z)&\delta_x^2f(z)&\dots&\delta_x^nf(z)\end{pmatrix}$$
satisfies 
$$\delta_x F=FA,$$
where $A$ satisfies \eqref {excl}.
However, in this case
$$(I+\alpha_\pm\delta_x)F=F(I+\alpha_\pm A),$$
and if $v$ is such that $(I+\alpha_\pm A)v=0,$ then, according to Proposition \ref{propfourmod},  $Fu=0.$ By minimality of $n,$ $u=0,$ hence   \eqref{excl} holds.
\end{proof}

\section{A realization for discrete analytic Schur functions}
\label{realsec}
We now give a characterization of discrete analytic Schur functions in terms of a realization.  Our approach is based on the $\epsilon$-method
of Krein and Langer (see e.g. \cite{kl1}),  that was adapted in \cite{adrs} using the underlying reproducing kernel Pontryagin spaces to find realizations of generalized Schur functions.

\begin{theorem}\label{schreal} Let $S(z)$ be a $\mathbf L(\mathfrak{H}_1,\mathfrak{H}_2)$-valued discrete analytic Schur function. Then the associated de Branges - Rovnyak space $\mathcal{H}(S)$ is translation-invariant: for every fixed $w\in\Lambda_+$ the translation
$$f(z)\mapsto f(z+w)$$
is a bounded linear operator on $\mathcal{H}(S).$
In particular,   $\mathcal{H}(S)$ is $\delta_x$-invariant,  and for every $f\in\mathcal{H}(S)$ the following estimate holds:
\begin{equation}\label{schineq} \|\delta_xf\|_{\mathcal{H}(S)}^2\leq \|f\|_{\mathcal{H}(S)}^2-\|f(0)\|_{\mathfrak{H}_2}^2.\end{equation}
Furthermore,
$S(z)$ can be represented in the form
\begin{equation}\label{schre0}S(z)=D+C(I-zA)^{-\odot}\odot(zB),\end{equation}
where
\begin{equation}\label{can1}\begin{pmatrix} A& B\\C&D\end{pmatrix}:\begin{pmatrix}   \mathcal{H}(S)\\   \mathfrak{H}_1\end{pmatrix}\longrightarrow\begin{pmatrix}
  \mathcal{H}(S)\\   \mathfrak{H}_2\end{pmatrix} \end{equation} is a coisometry defined by
\begin{equation}\label{can2}\begin{aligned}
Af(z)&=(\delta_x f)(z);&\quad  (B\lambda)(z)&=(\delta_xS)(z)\lambda;\\
Cf(z)&=f(0);&\quad D\lambda&=S(0)\lambda.
\end{aligned}\end{equation}
\end{theorem}
\begin{proof}
Consider the linear relation
$$R\subset\begin{pmatrix}
  \mathcal{H}(S)\\   \mathfrak{H}_2\end{pmatrix}\times \begin{pmatrix}   \mathcal{H}(S)\\   \mathfrak{H}_1\end{pmatrix}$$
which is spanned by the pairs 
$$\left(\begin{pmatrix} K^S(\cdot, w)\alpha\\ \beta\end{pmatrix},\begin{pmatrix} (K^S(\cdot, w+1)-K^S(\cdot,w))\alpha+K^S(\cdot,0)\beta\\ (\delta_x S)(w)^{*}\alpha+S(0)^{*}\beta\end{pmatrix}\right)$$
as $w$ ranges over $\Lambda_+,$ $\alpha$ and $\beta$ -- over $  \mathfrak{H}_2.$
The relation $R$ is isometric,  since
\begin{multline*}
\langle(K^S(\cdot, w_1+1)-K^S(\cdot,w_1))\alpha_1,(K^S(\cdot, w_2+1)-K^S(\cdot,w_2))\alpha_2\rangle_{  \mathcal{H}(S)}\\=
\langle(K^S(w_2+1,w_1+1)-K^S(w_2,w_1+1))\alpha_1,\alpha_2\rangle_{  \mathfrak{H}_2}-\\
-\langle(K^S(w_2+1,w_1)-K^S(w_2,w_1))\alpha_1,\alpha_2\rangle_{  \mathfrak{H}_2}\\
=\langle (\delta_xK^S(\cdot,w_1+1))(w_2)\alpha_1,\alpha_2\rangle_{  \mathfrak{H}_2}-
\langle (\delta_xK^S(\cdot,w_1))(w_2)\alpha_1,\alpha_2\rangle_{  \mathfrak{H}_2}\\
=\langle (ZK^S(\cdot,w_2))(w_1+1)^*\alpha_1,\alpha_2\rangle_{  \mathfrak{H}_2}-\langle (\delta_xS)(w_2)S(w_1+1)^*\alpha_1,\alpha_2\rangle_{  \mathfrak{H}_2}
\\-
\langle (ZK^S(\cdot,w_2))(w_1)^*\alpha_1,\alpha_2\rangle_{  \mathfrak{H}_2}+\langle (\delta_xS)(w_2)S(w_1)^*\alpha_1,\alpha_2\rangle_{  \mathfrak{H}_2}\\
=\langle K^S(w_2,w_1)\alpha_1,\alpha_2\rangle_{\mathfrak{H}_2}-[(\delta_xS)(w_1)^{*}\alpha_1,(\delta_xS)(w_2)^{*}\alpha_2\rangle_{  \mathfrak{H}_1};
\end{multline*}
\begin{multline*}
\langle K^S(\cdot,0)\beta_1,K^S(\cdot,0)\beta_2\rangle_{  \mathcal{H}(S)}=\langle K^S(0,0)\beta_1,\beta_2\rangle_{  \mathfrak{H}_2}\\
=\langle\beta_1,\beta_2\rangle_{  \mathfrak{H}_2}-\langle S(0)^{*}\beta_1,S(0)^{*}\beta_2\rangle_{  \mathfrak{H}_1};
\end{multline*}
\begin{multline*}
\langle(K^S(\cdot, w+1)-K^S(\cdot,w))\alpha,K^S(\cdot,0)\beta\rangle_{  \mathcal{H}(S)}=\langle(K^S(0, w+1)-K^S(0,w))\alpha,\beta\rangle_{  \mathfrak{H}_2}\\
=-\langle S(0)(\delta_xS)(w)^*\alpha,\beta\rangle_{  \mathfrak{H}_2}=-\langle (\delta_xS)(w)^*\alpha,S(0)^*\beta\rangle_{  \mathfrak{H}_1}
\end{multline*}
Since $R$ is an isometric relation with a dense domain,  it extends as the graph of an isometry
$$\begin{pmatrix} A'& C'\\B'&D'\end{pmatrix}:\begin{pmatrix}   \mathcal{H}(S)\\   \mathfrak{H}_2\end{pmatrix}\longrightarrow\begin{pmatrix}
  \mathcal{H}(S)\\   \mathfrak{H}_1\end{pmatrix} ;$$  the adjoint is a  coisometry \eqref{can1},  where $A,B,C,D$ are as in \eqref{can2}.
In particular,  $\begin{pmatrix}A\\C\end{pmatrix}$ is a contraction,  which implies \eqref{schineq}. Furthermore,  since $A=\delta_x$ is a contraction on $\mathcal{H}(S),$ it follows from Proposition \ref{propfourmod} that
$$I+\alpha_+\delta_y=(I+\alpha_-\delta_x)^{-1}$$ is a bounded operator on $  \mathcal{H}(S),$
and that $$I+\delta_y=(I+\alpha_+\delta_x)(I+\alpha_-\delta_x)^{-1}$$ is an invertible operator on $  \mathcal{H}(S).$  Since  $\mathcal{H}(S)$ is invariant under the action of operators $I+\delta_x$,  $I+\delta_y$ and $(I+\delta_y)^{-1}$,  it is translation-invariant.  
Finally,   for any $f(s)\in   \mathcal{H}(S)$  it holds that
\begin{multline*}((I-zA)^{-\odot}f)(s)=((I+\delta_x)^{\Re(z)}\left((I+\alpha_+\delta_x) (I+\alpha_-\delta_x)^{-1}\right)^{\Im(z)}f)(s)\\=((I+\delta_x)^{\Re(z)}(I+\delta_y)^{\Im(z)}f)(s)=f(s+z),
\end{multline*}
hence
$$C(I-zA)^{-\odot}f=f(z)$$ is the point evaluation operator,
and 
$$D+C(I-zA)^{-\odot}\odot (zB)=S(0)+(Z\delta_xS)(z)=S(z).$$
\end{proof}
 
Formula \eqref{schre0} associates to a discrete analytic Schur function  a certain canonical coisometric operator colligation.   As the following  result shows,   every coisometric colligation corresponds to a discrete analytic Schur function.

\begin{theorem}\label{charschur} Let $$S(z)=D+C(I-zA)^{-\odot}\odot(zB),$$
where
$$\begin{pmatrix} A& B\\C&D\end{pmatrix}:\begin{pmatrix} \mathcal H\\   \mathfrak{H}_1\end{pmatrix}\longrightarrow\begin{pmatrix}
\mathcal H\\   \mathfrak{H}_2\end{pmatrix} $$ is a coisometry; $\mathcal H$ is a Hilbert space.
Then 
$S(z)$ is a discrete analytic Schur function and 
$K^S(z,w)$ is of the form \begin{equation}\label{schker}
K^S(z,w)=C(I-zA)^{-\odot}(I-\bar wA^*)^{-\odot}C^*.\end{equation}
\begin{proof}
Denote
$$F(z)=C(I-zA)^{-\odot}.$$
Then $$S=D+(ZF)(z)B.$$
Since
$$(ZF)(z)A+C=C+C(I-zA)^{-\odot}\odot(zA)=C(I-zA)^{-\odot}\odot(I-zA+zA)=F(z),$$
one has
$$ \begin{pmatrix} (ZF)(z)&I\end{pmatrix}\begin{pmatrix} A\\ C\end{pmatrix}=F(z).$$
Furthermore, 
$$S(z)=\begin{pmatrix}(ZF)(z)&I\end{pmatrix}\begin{pmatrix}B\\ D\end{pmatrix}.$$
Therefore,
\begin{multline*}
(Z^nS)(z)(Z^n S)(w)^{*}\\=\begin{pmatrix}(Z^{n+1}F)(z)&z^{(n)}I\end{pmatrix}\begin{pmatrix}B\\ D\end{pmatrix}\begin{pmatrix}B^{*}& D^{*}\end{pmatrix}\begin{pmatrix}(Z^{n+1}F)(w)^{*} \\ \bar w^{(n)}I\end{pmatrix}\\
=\begin{pmatrix}(Z^{n+1}F)(z)&z^{(n)}I\end{pmatrix}\begin{pmatrix}(Z^{n+1}F)(w)^{*} \\ \bar w^{(n)}I\end{pmatrix}-\\
-\begin{pmatrix}(Z^{n+1}F)(z)&z^{(n)}I\end{pmatrix}\begin{pmatrix}A\\ C\end{pmatrix}\begin{pmatrix}A^{*}& C^{*}\end{pmatrix}\begin{pmatrix}(Z^{n+1}F)(w)^{*} \\ \bar w^{(n)}I\end{pmatrix}\\
=(Z^{n+1}F)(z)(Z^{n+1}F)(w)^{*}+z^{(n)}\bar w^{(n)}I-(Z^nF)(z)(Z^nF)(w)^{*}.
\end{multline*}
Since,  by Lemma \ref{zzero},  $(Z^nF)(z)\rightarrow 0$ pointwise as $n\rightarrow\infty,$ 
$$K^S(z,w)=F(z)F(w)^{*}$$ is a positive kernel.
\end{proof}
\end{theorem}

We close this section with a theorem of Beurling - Lax type,  characterizing de Branges -Rovnyak spaces in terms of translation-invariance.

\begin{theorem} \label{beurlax} Let $ \mathfrak{H}_2$ be a Hilbert space, and let $\mathcal{H}$ be a Hilbert space of $ \mathfrak{H}_2$-valued functions,  that are discrete analytic in $\Lambda_+$. Suppose that space
  $\mathcal H$  is $\delta_x$-invariant, and that estimate 
\begin{equation}\label{schineq2} 
\|\delta_xf\|_{\mathcal{H}}^2\leq \|f\|_{\mathcal{H}}^2-\|f(0)\|_{\mathfrak{H}_2}^2.
\end{equation}
holds for every $f\in\mathcal{H}.$ Then there is a Hilbert space $ \mathfrak{H}_1$ and a
  $\mathbf L( \mathfrak{H}_1,   \mathfrak{H}_2)$-valued discrete analytic Schur function $S(z),$ such that $\mathcal H=  \mathcal{H}(S).$
\end{theorem}
\begin{proof}
  Let $A$ and $C$ be defined as in \eqref{can2}. Then $A\in \mathbf L(\mathcal H),$ $C\in \mathbf L(\mathcal H,   \mathfrak{H}_2)$,
  and $M=\begin{pmatrix} A^*&C^*\end{pmatrix}\in \mathbf L(\mathcal H\oplus    \mathfrak{H}_2, \mathcal H)$ is a contraction.  Therefore,  there exists a Hilbert space $  \mathfrak{H}_1$ and
  $N\in \mathbf L(\mathcal H\oplus    \mathfrak{H}_2,  \mathfrak{H}_1),$ such that
  $$
  I_{\mathcal H\oplus   \mathfrak{H}_2}-M^*M=N^*N.$$
Write
$$N^*=\begin{pmatrix} B\\ D\end{pmatrix}\in\mathbf L( \mathfrak{H}_1,\mathcal H\oplus  \mathfrak{H}_2),$$
so that 
$$\begin{pmatrix} A&B\\C& D\end{pmatrix} $$ is a coisometry.    Then,  according to Theorem \ref{charschur}, $$S(z)=D+C(I-zA)^{-\odot}\odot(zB)$$ is a Schur function,  and  $K^S(z,w)$ is of the form \eqref{schker}.
On the other hand,  for each $z\in\Lambda_+,$ the operator $F(z)=C(I-zA)^{-\odot}$ is the operator of point evaluation at $z$. Since $F(z)\in \mathbf L(\mathcal H,\CC),$ 
$\mathcal H$ is a reproducing kernel Hilbert space, and its reproducing kernel is $F(z)F(w)^*=K^S(z,w)$. Thus $\mathcal H=   \mathcal{H}(S).$
\end{proof}

\section{The case of a mesh $h$}
\label{meshsec}

  We consider the lattice $\Lambda_h=h\mathbb Z+ih\mathbb Z$ with $h>0$ (one could more generally consider the lattice $\Lambda_{h,k}=h\mathbb Z+ik\mathbb Z$, with $h,k>0$)
A complex-valued function $f(z)$ defined on the lattice $\Lambda_h$ is {\sl discrete analytic} if
(see \cite{MR0013411})
  \[
    \dfrac{f(z+h(1+i))-f(z)}{h(1+i)}=\dfrac{f(z+h)-f(z+ih)}{h(1-i)},\quad z\in\Lambda.
  \]
  With $f(z)=u(x,y)+iv(x,y)$ we get
  \[
    \begin{split}
      \frac{u(x+h,y+h)+iv(x+h,y+h)-u(x,y)-iv(x,y)}{h(1+i)}&=\\
      &\hspace{-6cm}=
      \frac{u(x+h,y)+iv(x+h,y)-u(x,y+h)-iv(x,y+h)}{h(1-i)}
      \end{split}
    \]
and so we get 
\[
      v_x+v_y-i(u_x+u_y)=u_x-u_y+i(v_x-v_y)\quad{\rm as}\quad h\rightarrow 0
    \]
    and hence the Cauchy-Riemann equations.\\

Let $h>0$. When working in $\Lambda_h$ we set
\begin{eqnarray}
\delta_{x,h}f(z)&=&\frac{f(z+h)-f(z)}{h}\\
(Z_h f)(z)&=&\frac{(f(0)-f(z))h}{2}+\int_0^z f \delta z
\end{eqnarray}
We have
\[
 2\delta_{x,h}Z_h f(z)=f(z)-f(z+h)+(2/h)\int_z^{z+h}  f \delta z=2f(z)
\]

Let
    \begin{equation}
      x^{(n)}_h=\frac{x(x-h)\cdots (x-(n-1)h)}{n!}
    \end{equation}
    and $z^{(n)}$ the unique polynomial discrete analytic extension to $\Lambda_{+,h}$.
    We set
    \begin{equation}
K_h(z,w)=\sum_{n=0}^\infty z^{(n)}_h\overline{w^{(n)}_h}
     \end{equation} 
     and $\mathbf   \mathfrak{H}_2(\Lambda_{+,h})$ the associated reproducing kernel Hilbert space, which we will call Hardy space associated to the mesh $h$. The proof of the following results areas for $h=1$ and will be omitted.
     
\begin{lemma} In the associated Hardy space associated to $h$ we have
  \begin{equation}
    Z_h=\delta_{x,h}^*
    \end{equation}
  \end{lemma}

   \begin{lemma}
We have
    \begin{equation}
      f(x)=\sum_{n=0}^\infty (\delta_{x,h}^nf)(0)x_h^{(n)}
      \end{equation}
    \end{lemma}

     We see that
    \[
\lim_{h\rightarrow 0}x^{(n)}_h=\frac{x^n}{n!}
\]
and by discrete analytic polynomial extension
    \[
      \lim_{h\rightarrow 0}z^{(n)}_h=\frac{z^n}{n!}.
    \]

    \begin{proposition}
      The reproducing kernel of the Hardy space  $\mathbf   \mathfrak{H}_2(\Lambda_{+,h})$ tends to
      \begin{equation}
        \label{kl2}
K(z,w)=\sum_{n=0}^\infty\frac{z^n\overline{w}^n}{(n!)^2}.
\end{equation}
In that space it holds that
\begin{equation}
  \label{456}
(\partial^*f)(z)=\int_0^z f(s)ds
 \end{equation}
\end{proposition}

\begin{proof}
  We only prove \eqref{456}. In the inner product of  the reproducing kernel Hilbert space
  with reproducing kernel \eqref{kl2} we have on the one hand,
  \[
   \langle \partial z^n, z^m\rangle=n\langle z^{n-1},z^m\rangle=(m+1)(m!)^2\delta_{n-1,m},
\]
and on the other hand,
  \[
   \langle z^n, \frac{z^{m+1}}{m+1}\rangle=\frac{1}{m+1}\langle z^{n},z^{m+1}\rangle=\frac{1}{m+1}((m+1)!)^2\delta_{n,m+1},
 \]
 for $n,m\in\mathbb N_0$, and the usual conventions for $n=0$ in the first case.
 This ends the proof since $\frac{1}{m+1}((m+1)!)^2=(m+1)(m!)^2$.
\end{proof}

We note that, at least formally,

\[
  \delta_{x,h}\longrightarrow \partial
\]
and
\[
Z_hf(z)\longrightarrow \int_0^zf(s)ds
\]
and that, for every $h$, as already noted.
\[
\delta_{x,h}^*=Z_h,
  \]
which, as $h\rightarrow 0$ gives \eqref{456}.

\section*{Acknowledgments}
Daniel Alpay thanks the Foster G. and Mary McGaw Professorship in Mathematical Sciences, which supported this research.

\bibliographystyle{plain}
\def\cprime{$'$} \def\cprime{$'$} \def\cprime{$'$}
  \def\lfhook#1{\setbox0=\hbox{#1}{\ooalign{\hidewidth
  \lower1.5ex\hbox{'}\hidewidth\crcr\unhbox0}}} \def\cprime{$'$}
  \def\cprime{$'$} \def\cprime{$'$} \def\cprime{$'$} \def\cprime{$'$}
  \def\cprime{$'$}

\end{document}